\documentclass[english]{article}
\usepackage[T1]{fontenc}
\usepackage[latin9]{inputenc}
\usepackage{babel}
\usepackage{amsmath}
\usepackage{amsfonts}
\usepackage{amssymb}
\usepackage{amsthm}
\usepackage{graphicx}
\DeclareGraphicsExtensions{.pdf,.png,.jpg}

\newtheorem{prop}{Proposition}[subsection]

\newtheorem{cor}{Corollary}[subsection]

\newtheorem*{theorem*}{Theorem}

\theoremstyle{remark}
\newtheorem{remark}{Remark}[subsection]

\theoremstyle{definition}
\newtheorem{defin}{Definition}[subsection]
\newtheorem{example}{Example}[subsection]

\begin{document}

\title{The multi-variable Affine Index Polynomial for tangles}

\author{Nicolas Petit}

\maketitle
\begin{center}Boston College\\ petitnicola@gmail.com, nicolas.petit@bc.edu\end{center}

\begin{abstract}
\noindent The multi-variable affine index polynomial was defined by the author in previous work. The aim of this short note is to update the definition so it is generalizable to virtual tangles and to show it is compatible with tangle decomposition. We also introduce the Turaev moves for virtual tangles, and discuss how to recover the weight of each crossing as an intersection pairing of homology classes.

\end{abstract}

\section{Introduction}

The Affine Index Polynomial, first introduced in \cite{affineindexpolynomial}, is part of a family of virtual knot invariants related to the notion of an intersection index, and was generalized in \cite{virtualknotcobordismaffineindex} to compatible virtual links and then again in \cite{multivariableaip} to a multi-variable polynomial by the author. The aim of this paper is to wrap up some lingering questions from that previous work, extending the generalization to the case of virtual tangles in a manner that is compatible with the tangle decomposition of a virtual knot/link. In doing so we fill what seems to be a gap in the literature, introducing the notion of virtual Turaev moves that let us decompose a virtual tangle as a tensor and composition of simple generators.

This paper is structured as follows. Section \ref{background} contains some background on virtual knots and tangles, the definition of the virtual Turaev moves (Prop. \ref{prop4} and \ref{prop5}), and a brief summary of the author's previous work from \cite{multivariableaip}. Section \ref{maiptangles} discusses the updated definition, showing it is still a Vassiliev invariant of virtual tangles, while section \ref{homolosection} discusses how the weights can be recovered from the intersection pairing of some special homology classes related to the tangle.
\section{Background}
\label{background}
\subsection{Virtual knots and tangles}

This section is inspired by the treatment found in \cite{jacksonmoffatt}.

\begin{defin} 
A virtual knot is a 4-valent graph in which the edges come in two opposite pairs (the ``strands''), with each vertex decorated by a classical crossing (either positive or negative) or a virtual crossing, as shown in Figure \ref{fig1}, modulo the classical and virtual Reidemeister moves, pictured in Figures \ref{fig2} and \ref{fig3}. Virtual knots can also be interpreted as knots (i.e. immersions of $S^1$) in thickened surfaces, up to the classical Reidemeister moves and stabilization/destabilization. An orientation on the knot is a choice of direction to move along the knot; the sign of a crossing is well-defined regardless of the orientation picked.
A virtual knot with multiple components is called a virtual link. Crossings in a virtual link are divided between self-crossings, which are crossings in which both strands belong to the same component, and mixed crossings, where the two strands belong to different components. The sign of a self-crossing is well-defined, but the sign of a mixed crossing depends on the relative orientations of the components involved. When working with a virtual link, we usually assume it is oriented and a starting point for the orientation has been chosen arbitrarily.

\begin{figure}[!h]
\centering
\includegraphics[scale=0.1]{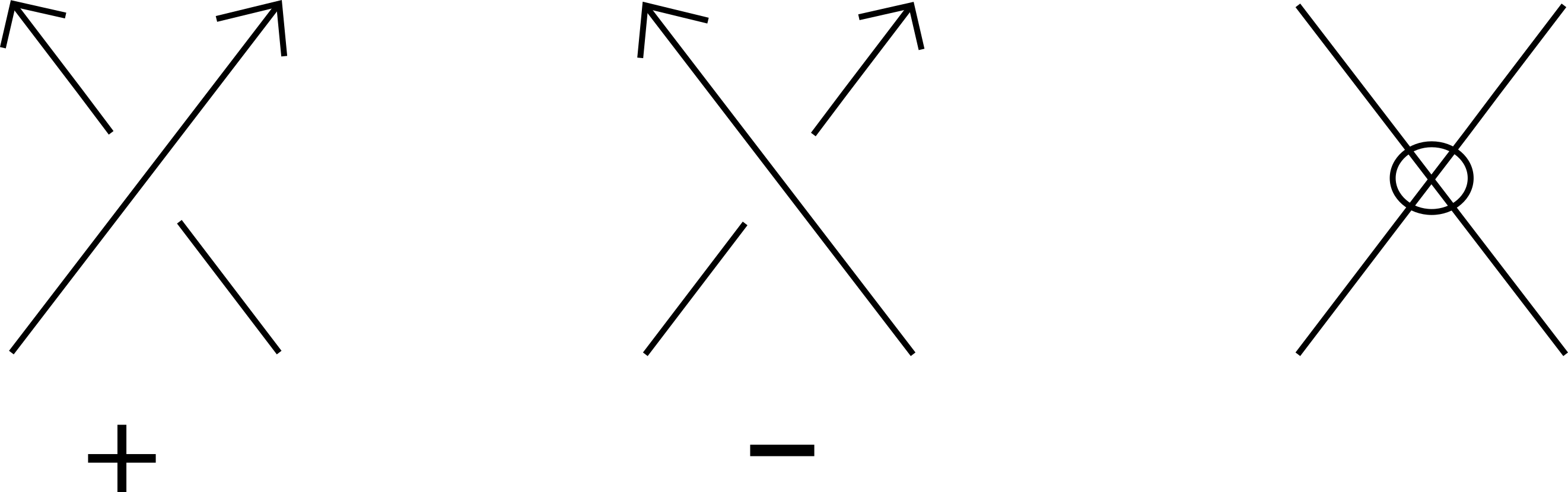}
\caption{The three crossing types}
\label{fig1}
\end{figure}

\begin{figure}[!h]
\centering
\includegraphics[scale=0.065]{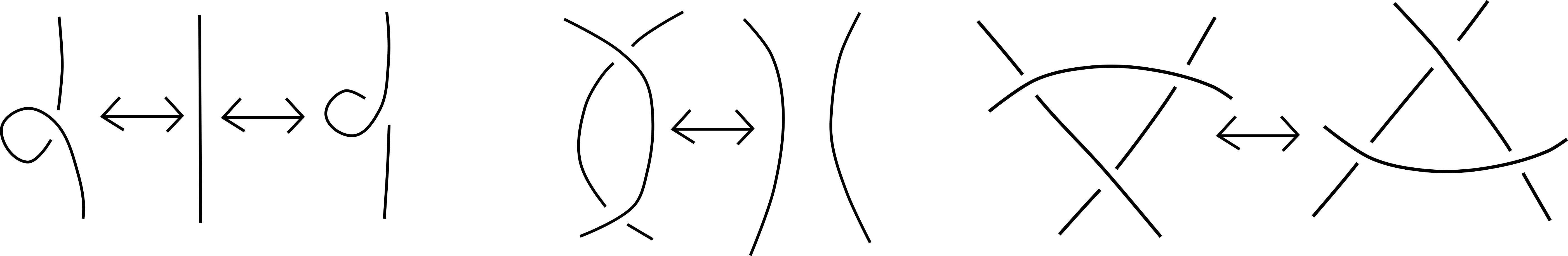}
\caption{The classical Reidemeister moves}
\label{fig2}
\end{figure}

\begin{figure}[!h]
\centering
\includegraphics[scale=0.065]{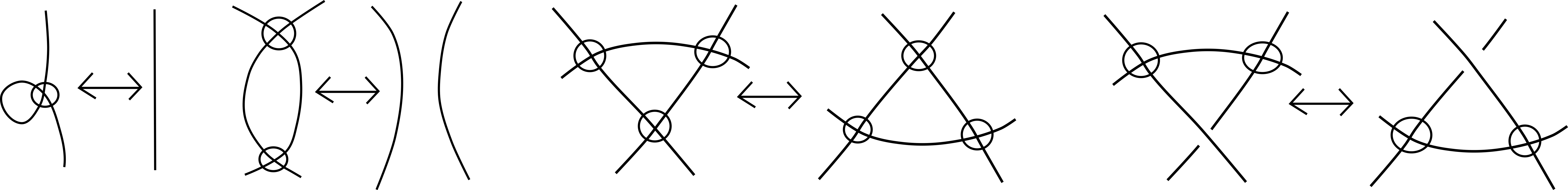}
\caption{The virtual Reidemeister moves}
\label{fig3}
\end{figure}

\end{defin}

\begin{defin}
A virtual tangle is a graph where every vertex is either 1-valent (which we consider as a fixed end) or 4-valent (with the crossing properties above), up to the classical and virtual Reidemeister moves performed away from the 1-valent vertices. A connected component either does not contain 1-valent vertices (in which case we call it closed) or it contains exactly two 1-valent vertices. A tangle whose components are all closed is a link, and a tangle with only one closed component is a knot, so the tangle category supersedes both knots and links. We usually portray our tangles vertically, with the 1-valent vertices glued to the top and/or the bottom of the square, see Figure \ref{fig4}. We say the tangle is an $(m,n)$-tangle if there are $m$ tangle ends glued to the top boundary and $n$ tangle ends glued to the bottom boundary.

Virtual tangles can also be realized as embeddings of a disjoint union of copies of $S^1$ (the closed components) and $I$ (the long components) into a thickened surface with boundary $F$, so that the ends of any long component belong to $\partial F$, up to the Reidemeister moves and stabilization/destabilization.
\end{defin}

\begin{figure}[!h]
\centering
\includegraphics[scale=0.15]{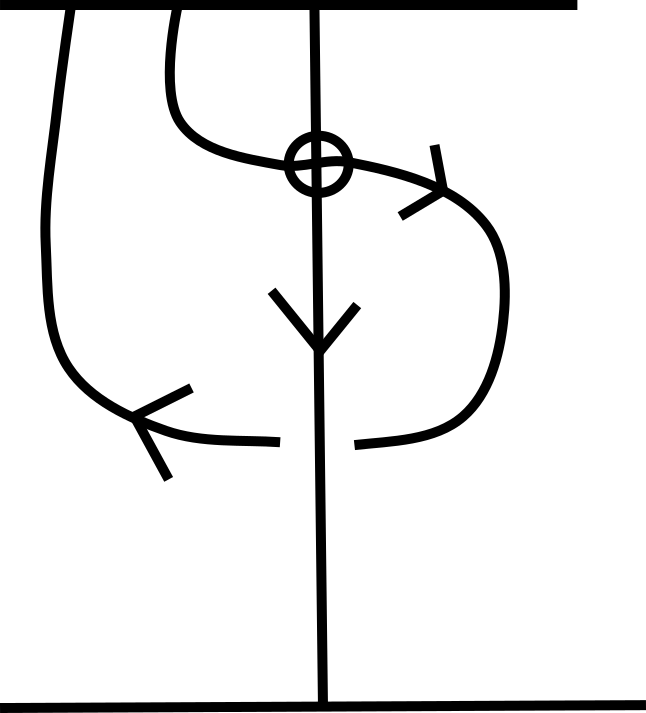}
\caption{An example of a $(3,1)$ virtual tangle. This particular tangle has two components and is also oriented.}
\label{fig4}
\end{figure}

There are two operations that can be performed on tangles, called tensor product and composition. The tensor product of the $(m,n)$-tangle $T$ with the $(p, q)$-tangle $T'$ is the $(m+p, n+q)$-tangle $T\otimes T'$ where we simply put the tangle $T'$ to the right of the tangle $T$. Now suppose that $T$ is an $(m,n)$-tangle and $T'$ is an $(n,p)$-tangle. The composition $T\circ T'$ is the $(m,p)$-tangle obtained by stacking tangle $T$ above tangle $T'$, connecting the ends appropriately. If the tangles $T, T'$ are oriented, the composition is only performed if the resulting tangle has a consistent orientation inherited from $T,T'$. 

We say that a tangle $T$ is generated by a certain set $S$ if we can reconstruct $T$ by composing and tensoring elements of $S$.

\begin{defin}
\label{def1}
For unoriented virtual tangles, consider the set shown in Fig. \ref{fig5}. Clearly any oriented tangle (possibly after isotopy) can be decomposed as the tensor product and composition of the basic pieces from the set.
For oriented virtual tangles we'll similarly use the set portrayed in Fig. \ref{fig6}.
\begin{figure}[!h]
\centering
\includegraphics[scale=0.15]{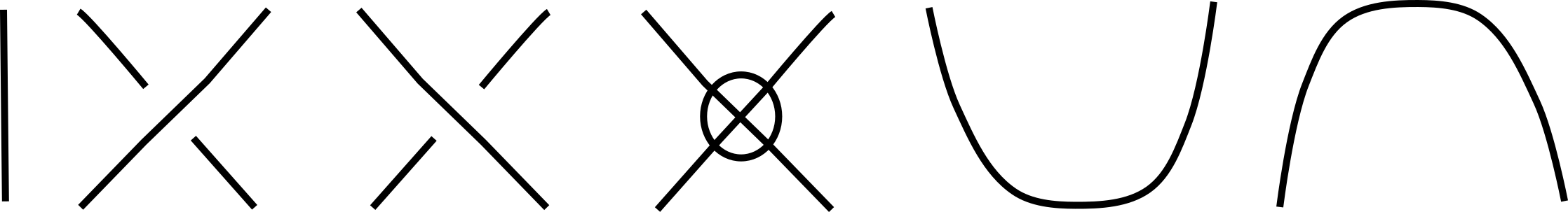}
\caption{The unoriented tangle generators.}
\label{fig5}
\end{figure}
\begin{figure}[!h]
\centering
\includegraphics[scale=0.12]{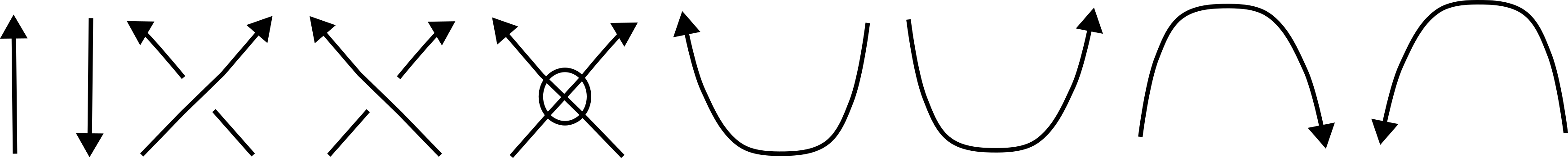}
\caption{The oriented tangle generators.}
\label{fig6}
\end{figure}
\end{defin}

\begin{prop}[The virtual Turaev moves]
\label{prop4}
The category of virtual tangles up to classical and virtual Reidemeister move is equivalent to the category of tangles generated by the unoriented set from Def. \ref{def1} up to the moves from Fig. \ref{fig7}, called the virtual Turaev moves.
\begin{figure}[!h]
\centering
\includegraphics[scale=0.07]{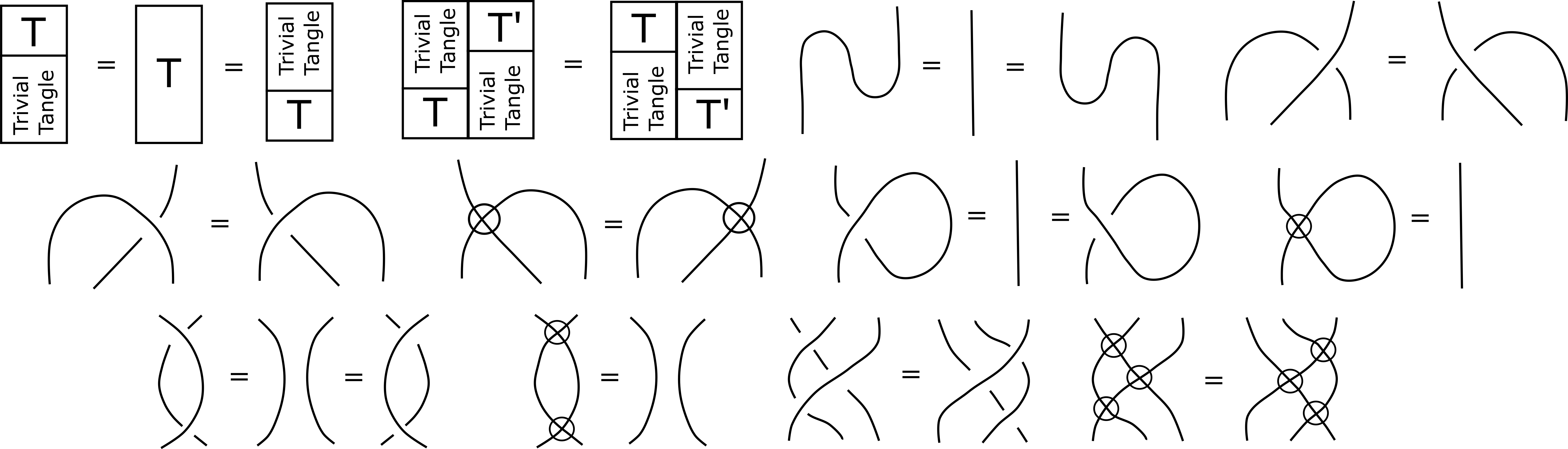}
\caption{The unoriented Turaev moves.}
\label{fig7}
\end{figure}
\end{prop}

\begin{prop}[The oriented virtual Turaev moves]
\label{prop5}
The category of oriented virtual tangles up to classical and virtual Reidemeister move is equivalent to the category of oriented tangles generated by the oriented set from Def. \ref{def1} up to the moves from Fig. \ref{fig8}, called the oriented virtual Turaev moves.
\begin{figure}[!h]
\centering
\includegraphics[scale=0.07]{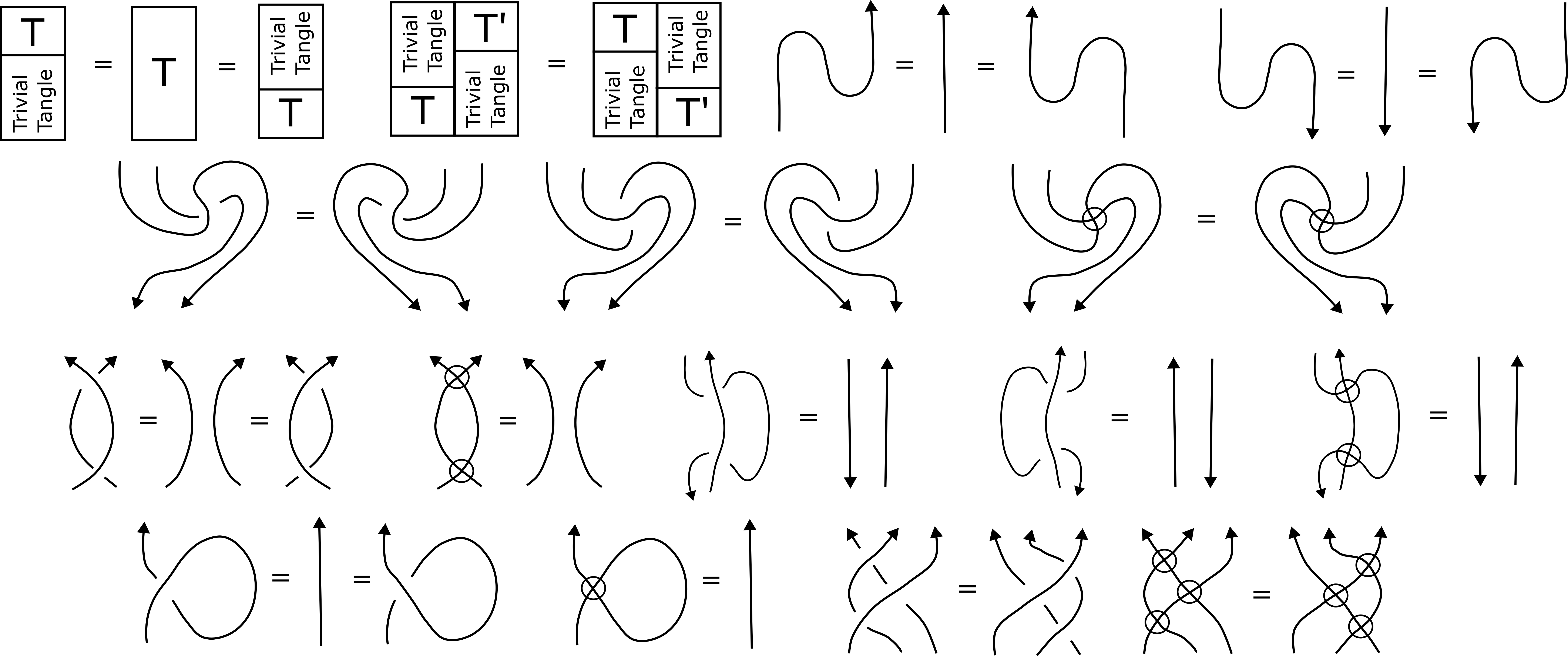}
\caption{The oriented Turaev moves.}
\label{fig8}
\end{figure}
\end{prop}

\begin{remark}
In the case of classical tangles, the Turaev moves were introduced in \cite{turaevtanglemoves}, but as far as the author knows no reference exists for the Turaev moves in a virtual tangle setting.
\end{remark}

\begin{proof}[Proof of Props. \ref{prop4} and \ref{prop5}]
The Turaev moves are simply taking a generating set of Reidemeister moves and adapting them so they are compatible with level-preserving isotopies. The first group of moves covers the planar isotopies that are not level-preserving, the second group represents the classical Reidemeister moves, and the final group covers the virtual Reidemeister moves.
\end{proof}

\begin{remark}\label{homolopairing}Since virtual tangles can be realized as (classical) tangles in a thickened surface with boundary $F$, we can look at the generic projection of the tangle onto $F$, in which case each component is a homology cycle and any two components intersect transversally. This means we can use parts of the tangle as inputs in the pairing of two transverse homology cycles $\alpha, \beta\in H_1(F, \partial F)$. This pairing, denoted by $\alpha\cdot \beta$, simply looks at whether at points where $\alpha, \beta$ intersect on $F$ the orientation $(\alpha, \beta)$ matches the orientation induced by $F$, and assigns $\pm1$ to every crossing depending whether the orientation matches ($+$) or does not match ($-$) the induced orientation.

\begin{figure}[!h]
\centering
\includegraphics[scale=0.12]{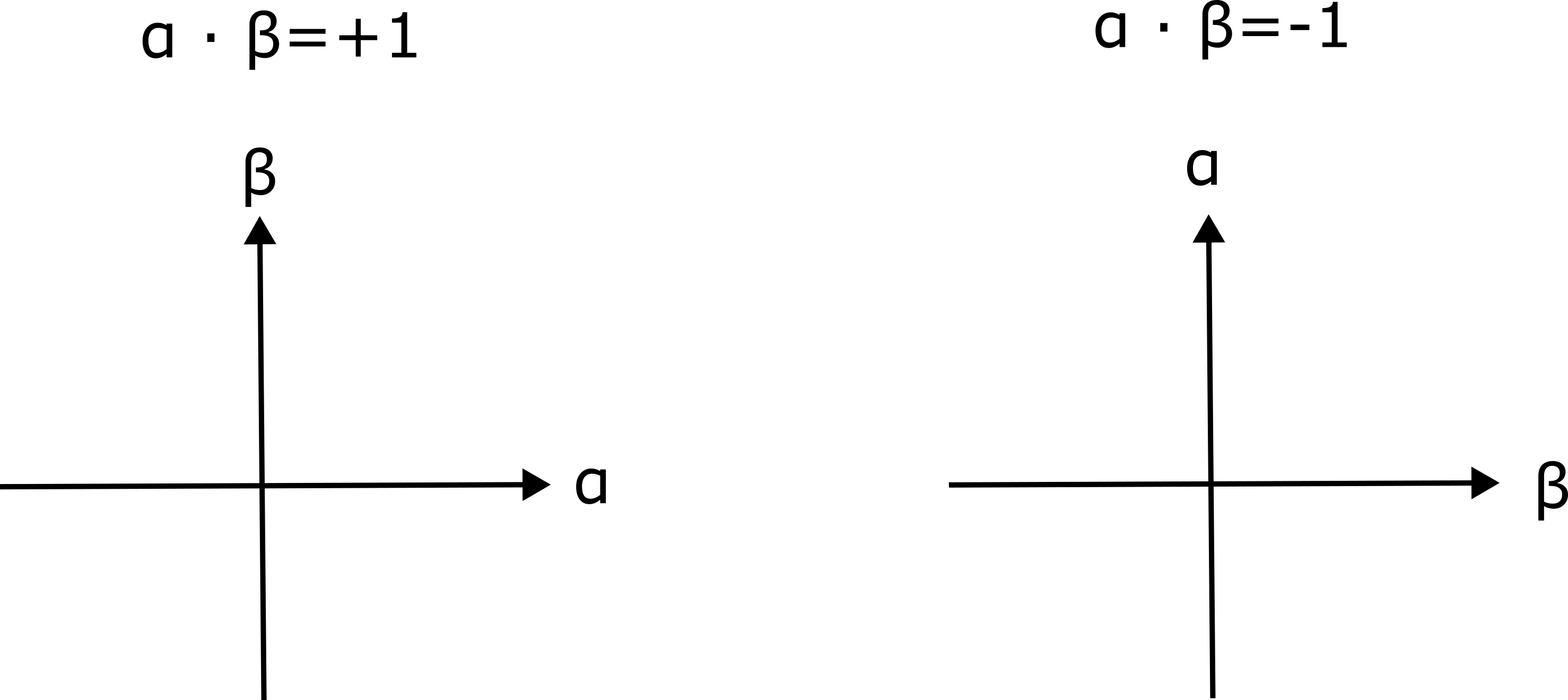}
\caption{The homology pairing $\alpha\cdot \beta$.}
\label{homolopairingfig}
\end{figure}
\end{remark}
\subsection{Vassiliev invariants}

\begin{defin}
Given a virtual knot invariant $\nu$, we can extend it to knots with double points by defining $\nu(D_s)=\nu(D_+)-\nu(D_-)$, where $D_s$ is the knot diagram with the singular point, $D_+$ is the diagram where the singular point was resolved as a positive crossing, and $D_-$ the diagram where the singular point was resolved as a negative crossing. If $D$ has more than one double point, extend the definition recursively, resolving every double point of $D$ in either the positive or negative way, keeping track of how many negative resolutions we chose to account for the sign in front of the resolution.

We say $\nu$ is a finite-type invariant or Vassiliev invariant (FTI for short) of order $\leq n$ if it vanishes identically on any knot with more than $n$ double points. Note that the recursive definition implies that if a FTI vanishes on all diagrams with $n+1$ double points, it will also vanish on all diagrams with more than $n+1$ double points.
\end{defin}

\begin{remark}
Singular virtual knots are virtual knot diagrams with an extra type of crossing decoration (the singular points), and inherit a couple of additional Reidemeister-style moves called the rigid vertex isotopy moves. We won't work with singular virtual knots much, so we send the interested reader to section 7 of \cite{virtualknottheory}.
\end{remark}

While Vassiliev invariants have been completely classified for classical knots, very little is known about their overall structure for virtual knots. We have plenty of examples of order one invariants, mainly coming from index polynomials of various kinds, and power series expansions of invariants have produced higher order Vassiliev invariants (see again \cite{virtualknottheory}); but compared to the finite-dimensional invariant spaces of the classical case, much is still to be explored.

\subsection{The Affine Index Polynomial and its generalization}

The Affine Index Polynomial (AIP for short) was first introduced in \cite{affineindexpolynomial}, generalized to compatible virtual links in \cite{virtualknotcobordismaffineindex}, and finally further extended by the author in \cite{multivariableaip} to a multi-variable polynomial. We will focus on the latter in this background section.

\begin{defin}
Let $L$ be an oriented link with $n$ ordered components $L_1,\ldots, L_n$. From the starting point used to define the orientation on each component, propagate a bilabel along the strands by starting it at some arbitrary value $(a_{i1}, a_{i2})$ and using the following rules to determine how the bilabel changes. If the link is not compatible (i.e. the algebraic intersection number of some of the components is nonzero) the label will not neatly wrap-up, but we don't consider that an issue.

\begin{figure}[!h]
\centering
\includegraphics[scale=0.1]{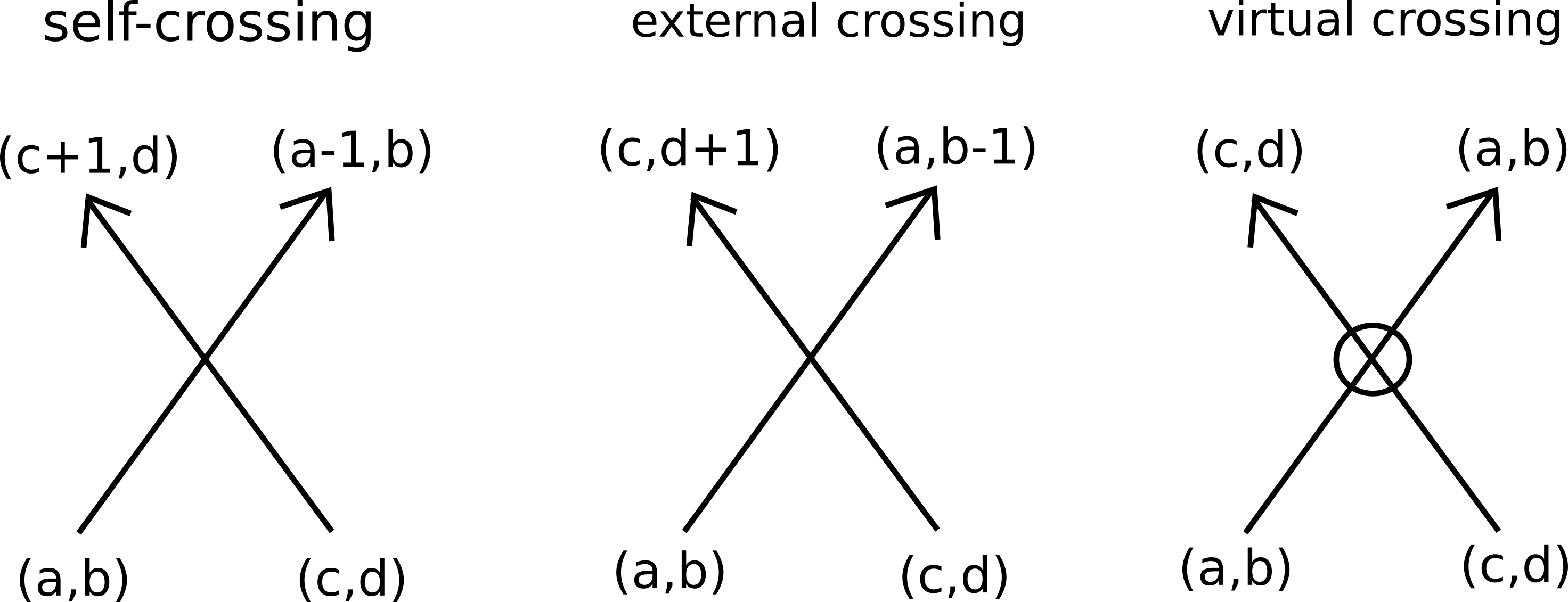}
\caption{The rules to determine the affine bilabeling.}
\label{fig9}
\end{figure}
\end{defin}

\begin{defin}
Given an oriented, ordered virtual link $L$, assume a classical crossing $d$ looks like either of the ones in Fig. \ref{fig9}. Assign to $d$ the weight $W_+(d)=(a+b)-(c+d)-1$ if $d$ is positive, and $W_-(d)=(c+d)-(a+b)+1$ if $d$ is negative (we don't distinguish between self-crossings and mixed ones). Now define the multi-variable AIP as $$p_L(t_1, \ldots, t_n)=\sum_{i=1}^n\ \sum_{c\in L_i\text{ over }L_j}sgn(c)(t_i^{W_{\pm}(c)}-1)$$ where the sum is over all crossings $c$ where component $L_i$ is the overstrand (including self-crossings of $L_i$).
\end{defin}

\begin{remark}
A couple of easy mnemonics to remember how to find the weight: you collapse the bilabel $(a,b)\mapsto a+b$, then take the incoming overstrand label minus the outgoing understrand label. Or, it's the incoming overstrand label minus the incoming understrand label minus the sign of the crossing.
\end{remark}

\begin{remark}
If we replace all the variables $t_i$ with the one variable $t$ and work with a compatible link, the multi-variable AIP reduces to the invariant described in \cite{virtualknotcobordismaffineindex}, whereas for a knot we recover the original AIP from \cite{affineindexpolynomial}.
\end{remark}

\begin{prop}[\cite{multivariableaip}]
The multi-variable AIP is an invariant of the pair $(L,C)$, where $L$ is a link and $C$ is the affine bilabeling described above. It is also an order one Vassiliev invariants for virtual knots and links.
\end{prop}

\section{The tangle generalization}
\subsection{Definition and properties}
\label{maiptangles}
\begin{defin}
Let $T$ be a virtual tangle with $n$ ordered components $T_1, \ldots, T_n$. Pick an orientation for each component, and a starting point for each closed component. Color the strands of the tangle by starting each component $T_i$ with the label $c_i$, and propagating the labels as if they were an intersection index (see Fig. \ref{affinelabeling}). We call this the affine labeling of the tangle. Let $\delta_i$ be the difference between the final label of the component and its initial label; we call this the index difference of component $i$

\begin{figure}[!h]
\centering
\includegraphics[scale=0.1]{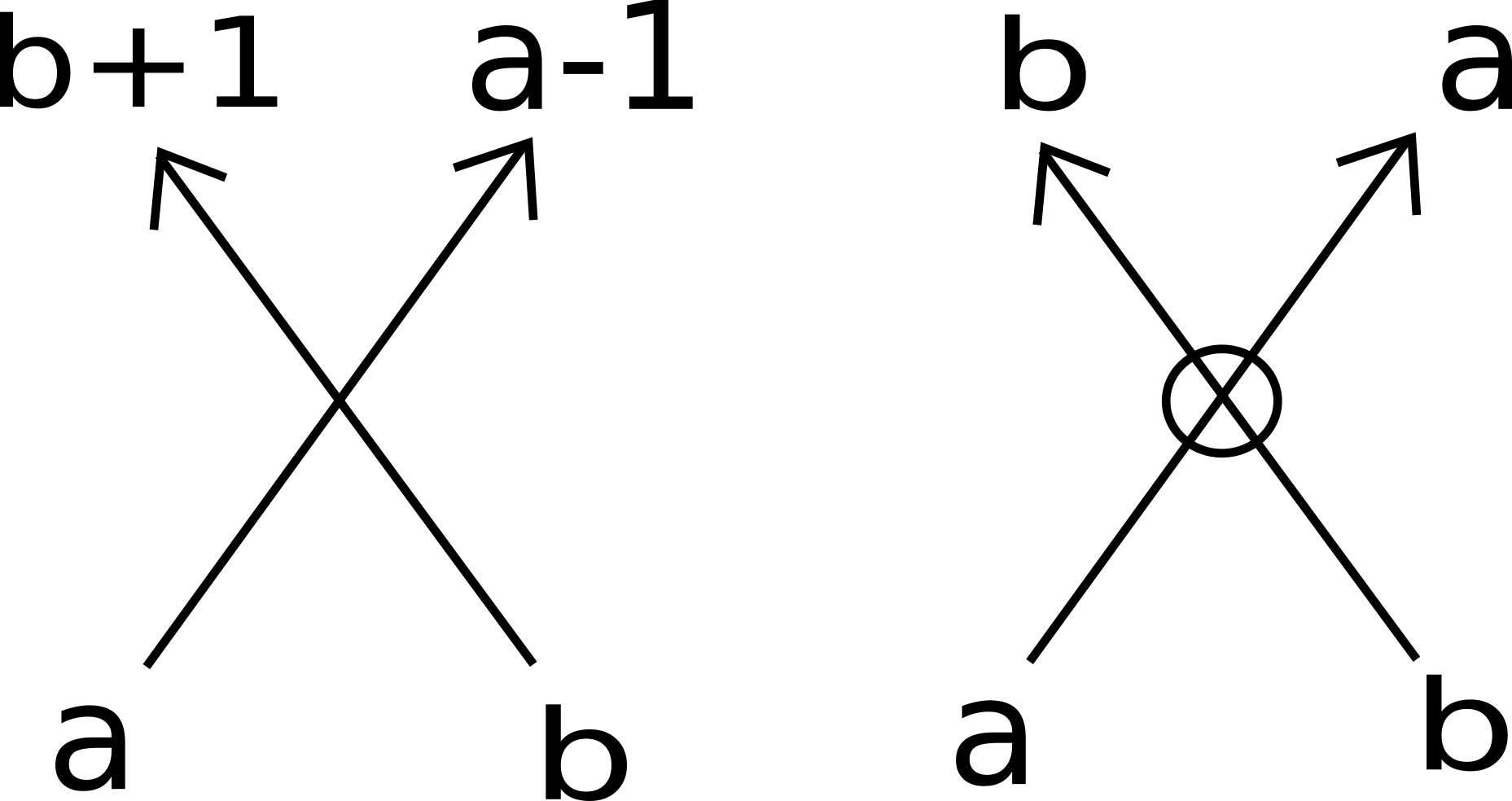}
\caption{The rules to determine the affine labeling of our tangle.}
\label{affinelabeling}
\end{figure}

We then assign to every classical crossing $d$ the weight $W(d)=a-b-sgn(c)$ where $a$ is the label of the overstrand and $b$ is the label of the understand. Finally, we build a polynomial in the variables $t_1, \ldots, t_n$ out of these weights via the formula
$$p_T(t_1, \ldots, t_n)=\sum_{i=1}^n\ \sum_{d\in T_i\text{ over }T_j}sgn(c)t_i^{\delta_j}(t_i^{W(d)}-1)$$ where the index of the variable $t_i$ matches the index of the overstrand of the crossing, $\delta_j$ is the index difference of the understrand, and the sum is taken over all classical crossings where $T_i$ is an overstrand (including the case $i=j$). The resulting expression is called the Multi-Variable Affine Index Polynomial (or MAIP for short) for tangles.
\end{defin}

\begin{remark}
Compared to \cite{multivariableaip} we chose to condense the bilabel to a single label. Ultimately the label gets collapsed to compute the invariant anyway (and doesn't even need to be defined when using the intersection index method described below), so its only utility was to distinguish the change in index produced by the self-crossings from the one produced by the mixed crossings, which is of no use in the current paper. The rest of the results below hold if we replace $c_i$ with $(c_{i1}, c_{i2})$, use the rules from \cite{multivariableaip}, and then define $c_i=c_{i1}+c_{i2}$.

The other major change is the inclusion of the terms $\delta_j$ in the weights; for knots and compatible links we have $\delta_j=0$, so the definition matches the one in \cite{multivariableaip}, whereas long components and non-compatible links have a slight different definition. The reason for this adjustment will be explained in Prop. \ref{prop2}
\end{remark}

\begin{example}
\label{ex1}
\begin{figure}[!h]
\centering
\includegraphics[scale=0.17]{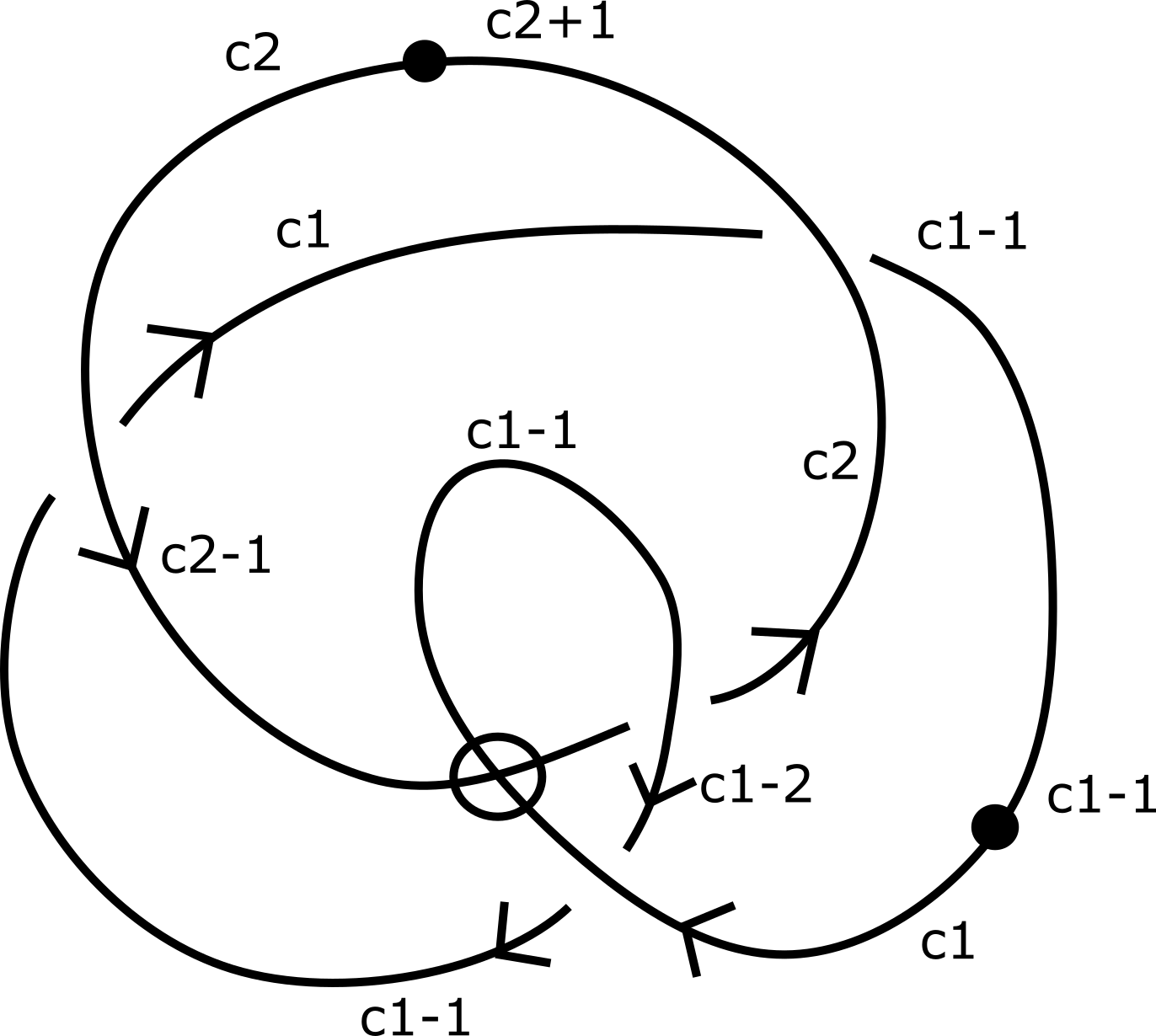}
\caption{An example of a tangle with only closed components.}
\label{example1}
\end{figure}

The tangle (really, non-compatible link) from Fig. \ref{example1} was already discussed in \cite{multivariableaip} in which it was colored with a bilabel. With our updated coloring and rules, we have $\delta_1=-1$ and $\delta_2=1$. Looking at the crossings, there are two crossings in which component two is the overstrand, contributing $$t_2^{\delta_1}(t_2^{c_2-c_1}-1)-t_2^{\delta_1}(t_2^{c_2-(c_1-1)}-1)$$
whereas component one is the overstrand for the two other crossings, contributing $$t_1^{\delta_1}(t_1^{c_1-(c_1-1)}-1)+t_1^{\delta_2}(t_1^{c_1-1-c_2}-1)$$
The overall contributions simplify to $$t_2^{c_2-c_1-1}-t_2^{c_2-c_1}+1-t_1^{-1}+t_1^{c_1-c_2}-t_1$$
\end{example}

\begin{figure}[!h]
\centering
\includegraphics[scale=0.15]{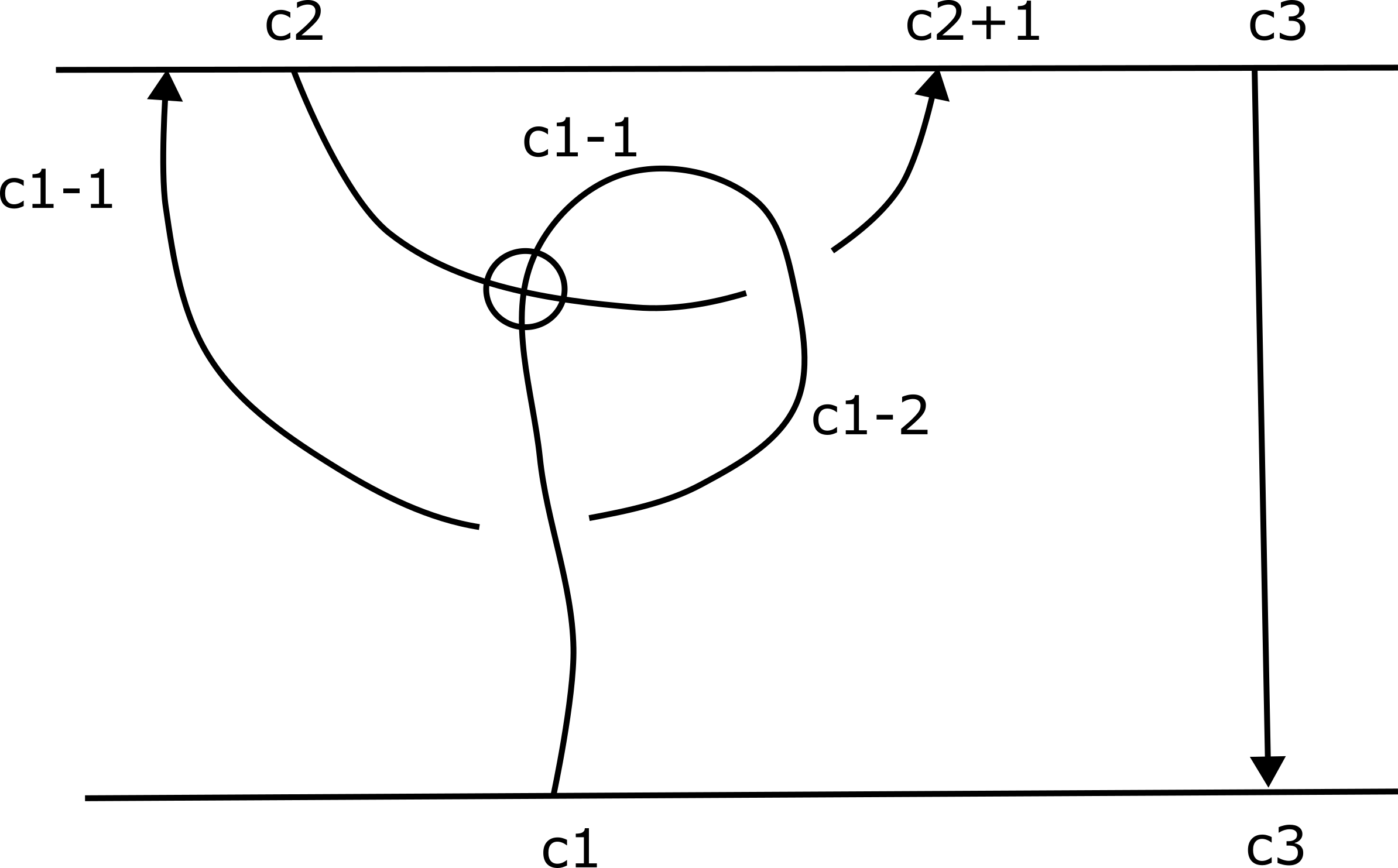}
\caption{Another example, with three separate components.}
\label{example2}
\end{figure}

\begin{example}
\label{ex2}
The tangle from Fig. \ref{example2} has $\delta_1=-1$, $\delta_2=1$, and $\delta_3=0$. The polynomial is $$t_1^{\delta_1}(t_1^{c_1-(c_1-1)}-1)+t_1^{\delta_2}(t_1^{(c_1-1)-(c_2+1)}-1)$$ which simplifies to
$$t_1^0-t_1^{-1}+t_1^{c_1-c_2-2}-t_1$$
\end{example}

\begin{example}
\label{ex3}
The tangle from Fig. \ref{example3} has (coincidentally) $\delta_1=-1$, $\delta_2=1$, and $\delta_3=0$. The polynomial is $$t_1^{\delta_3}(t_1^{c_1-(c_3+1)}-1)-t_2^{\delta_3}(t_2^{c_2-c_3}-1)$$ which simplifies to
$$t_1^{c_1-c_3-1}-1-t_2^{c_3-c_2}+1=t_1^{c_1-c_3-1}-t_2^{c_2-c_3}$$
\end{example}

\begin{figure}[!h]
\centering
\includegraphics[scale=0.15]{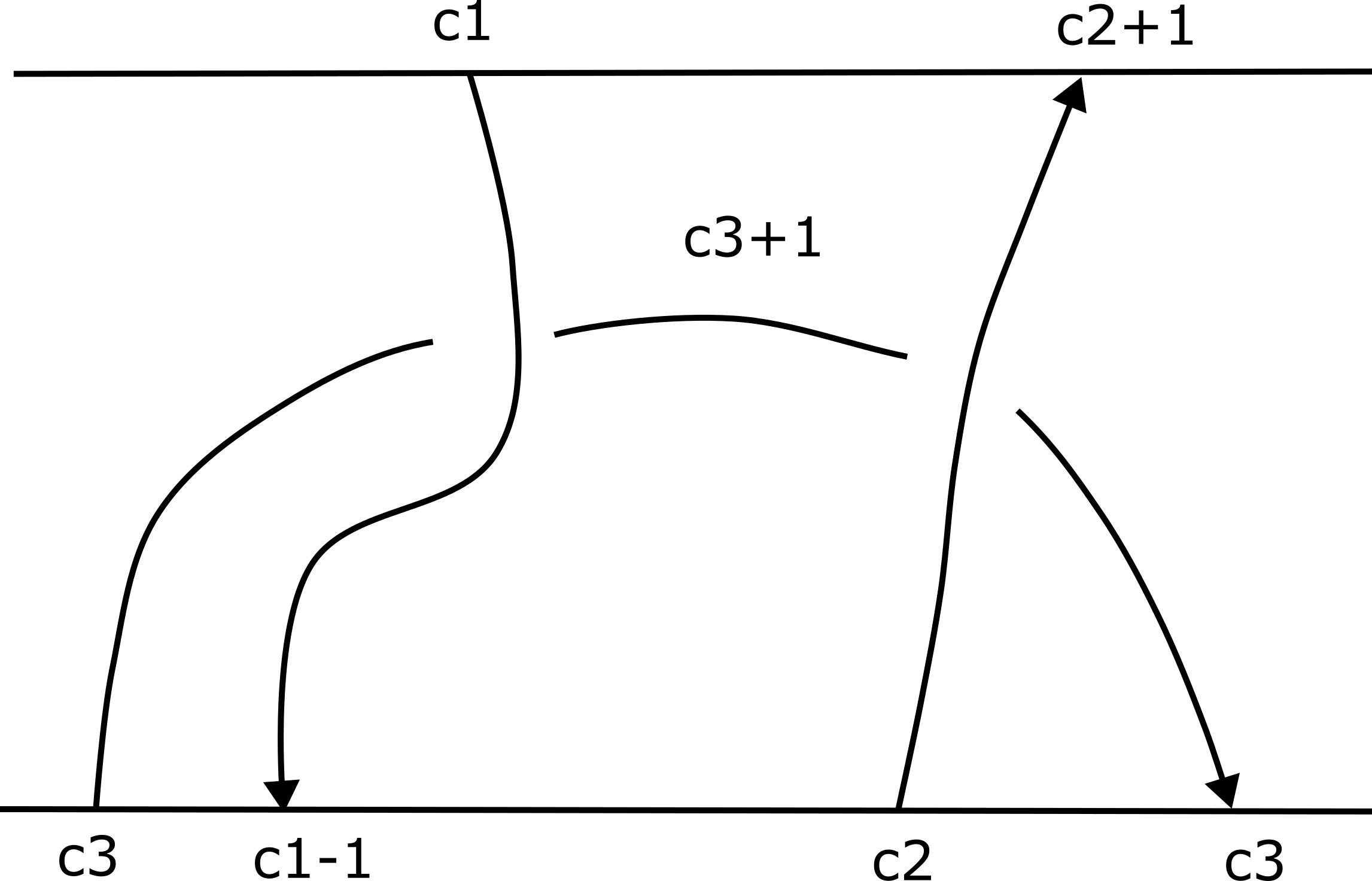}
\caption{One final example.}
\label{example3}
\end{figure}

\begin{prop}
\label{prop1}
The MAIP is an invariant of the pair $(T, C)$, where $T$ is a virtual tangle and $C$ is an affine labeling described above. It generalizes the invariants for knots and tangles from \cite{virtualknotcobordismaffineindex}. Moreover, it is an order one Vassiliev invariant for virtual tangles.
\end{prop}

\begin{proof}[Proof of \ref{prop1}]
To show that the MAIP is an invariant we need to check how it behaves under Reidemeister moves. In the case of a Reidemeister move one, the crossing has to be a self-crossing of component $i$, and the weights on either side of the crossing are the same. So the contribution to the invariant is $\pm t_i^{\delta_i}(t_i^0-1)=0$.

For Reidemeister move two, we will focus on the case where the two strands belong to different components. The overstrand at both crossings belongs to the same component (thus so does the understrand), and the weights are the same at both crossings ($b-a-1$ for the first R2 in Fig. \ref{R1R2affine}, $b-a+1$ for the other), so the overall contribution is $+t_i^{\delta_j}(t_i^w-1)-t_i^{\delta_j}(t_i^w-1)=0$, showing invariance under this move as well. The same argument works when both strands belong to the same component, simply replacing $\delta_j$ with $\delta_i$.

\begin{figure}[!h]
\centering
\includegraphics[scale=0.1]{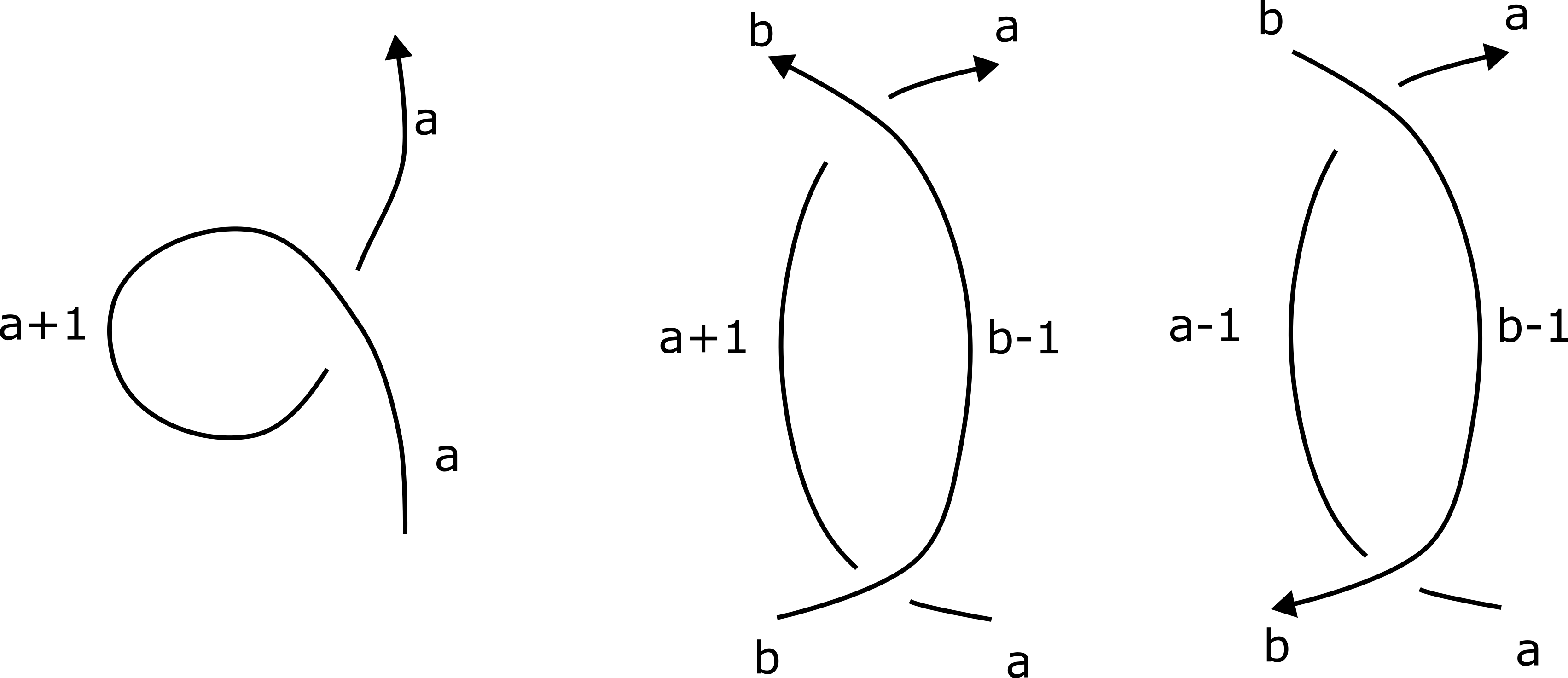}
\caption{The Reidemeister moves one and two and their labeling.}
\label{R1R2affine}
\end{figure}

The mixed Reidemeister move is easily taken care of, with both pictures having one crossing of weight $w=b-c-1$, so the overall contribution is $+t_i^{\delta_j}(t_i^w-1)$ (for the case pictured). Again, setting $j=i$ yields the proof of the case when both strands belong to component $i$.

Finally, let's look at the case for Reidemeister move three when all three strands belong to different components. There is a correspondence between the crossings, easily seen by which strands meet at a crossing. The weights, overstrands, understands, and crossing signs are the same on both sides of the correspondence, so the contribution on both sides of the move is the same. If two or three of the strands belong to the same component the argument still holds by setting the appropriate indices equal to each other. This concludes the proof that the MAIP is an invariant.

\begin{figure}[!h]
\centering
\includegraphics[scale=0.08]{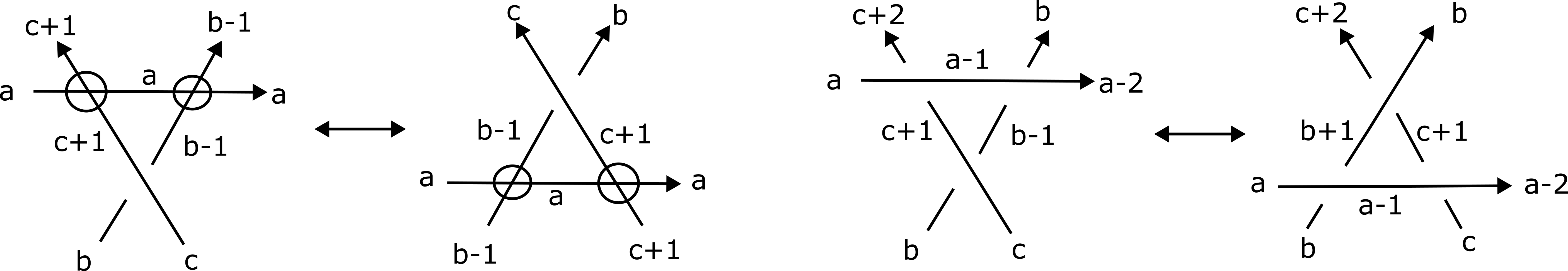}
\caption{Affine labeling for Reidemeister move three and the mixed one.}
\label{RmR3affine}
\end{figure}

For knots and compatible links described in \cite{virtualknotcobordismaffineindex} we have $\delta_i=0$ for all $i$, so our definitions reduce to the invariants described in said paper.

\begin{figure}[!h]
\centering
\includegraphics[scale=0.1]{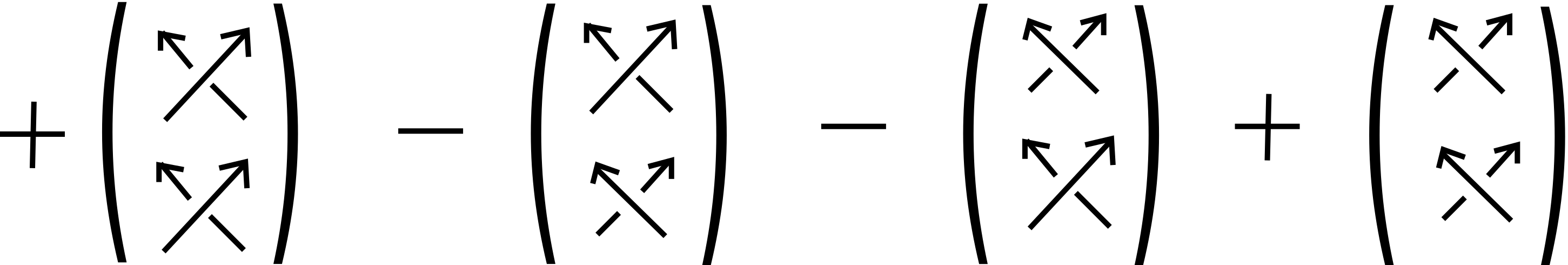}
\caption{A schematic representation of the resolution of the double points $d, d'$.}
\label{ftilessthan2}
\end{figure}

Showing it is a Vassiliev invariant is analogous to how it was proven in the previous paper by the author. Given two double points $d, d'$, we get the weighted resolution $D_{++}-D_{+-}-D_{-+}+D_{--}$, where $\pm\pm$ indicates how crossings $d,d'$ get respectively resolved, see Fig. \ref{ftilessthan2} for a schematic representation. But the rest of the diagram doesn't change, the labeling will be the same in all four terms as it is really only dependent on the flat knot underlying the diagram, and the contributions of crossings $d,d'$ don't influence each other (or the contribution of any other crossing) and get added together. So the contribution of a positive $d$ gets both added (in $D_{++}$) and subtracted (in $-D_{-+})$ for a net contribution of zero, and the same is true for the negative resolution of $d$ and both resolutions of $d'$. But every other crossing in $D$ appears twice with a $+$ sign and twice with a $-$ sign in the weighted sum, so their overall contribution is also zero, showing that, regardless of what $D$ looks like, if it has two double points that get resolved in the Vassiliev way they the MAIP of such combination is always zero. This shows that the MAIP is a Vassiliev invariant of order $< 2$.

It is very easy to provide an example of a tangle with one double point whose weighted resolution is nonzero, proving that the MAIP is a Vassiliev invariant of order $1$. In the example shown in Fig. \ref{vassilievtangleexample}, the two resolutions of the double point yield the polynomial $$t_1(t_1^{c_1-c_2-1}-1)-t_2^{-1}(t_2^{c_2-c_1+1}-1)=t_1^{c_1-c_2}-t_2^{c_2-c_1}+t_2^{-1}-t_1$$

\begin{figure}[!h]
\centering
\includegraphics[scale=0.1]{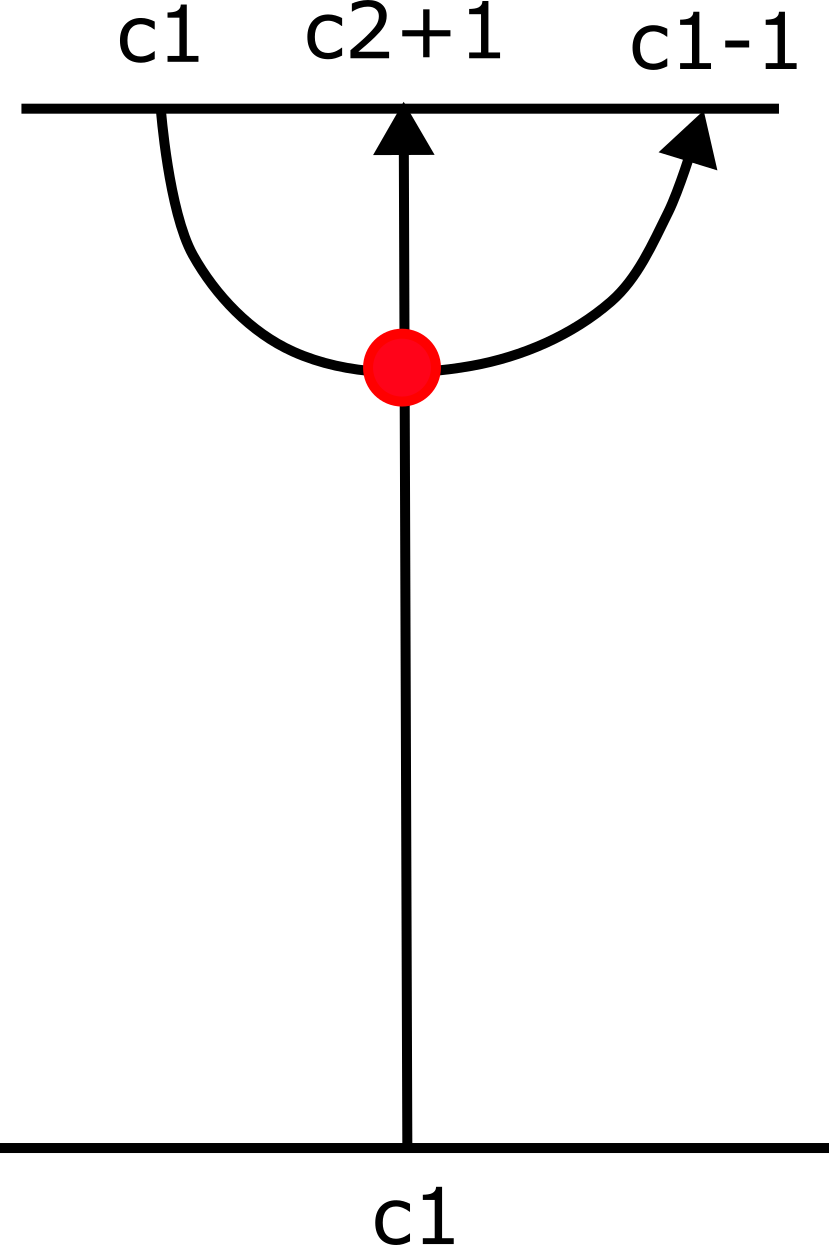}
\caption{A singular virtual tangle on which the MAIP is nonzero. The component starting with $c_1$ is component one, and the double point is represented as a red dot.}
\label{vassilievtangleexample}
\end{figure}
\end{proof}

\begin{prop}
\label{prop3}
The MAIP for tangles is well-behaved under tangle composition: $p_{T\otimes T'}=p_{T}+p_{T'}$ for any two tangles, and $p_{T\circ T'}=p_{T}+p_{T'}$ up to a couple of labeling adjustments:
\begin{itemize}
\item If the end of $T_i$ gets glued to the beginning of $T'_j$, the variables $t_j'$ should be replaced with $t_i$, and if the resulting glueing doesn't close a component the ending label of $T_i$ should be set equal to the starting label of $T'_j$: $c_i+\delta_i=c_j'$, and every occurrence of $c_j'$ in the exponents should be replaced with $c_i+\delta_i$.
\item If $T_{i1}$ glues to $T'_{j1}$, $T'_{j1}$ glues to $T_{i2}$, $T_{i2}$ glues to $T'_{j2}$,\ldots, up to $T_{in}$ glueing to $T'_{j_n}$ to overall form the component $T''_i$,  we must replace each of $\delta_{i1}, \ldots, \delta_{in}, \delta'_{j1}, \ldots, \delta'_{jn}$ with the value $$\delta''_{i}=\sum_{k=1}^n \delta_{ik}+\delta'_{jk}=\delta_{i1}+\delta'_{j1}+\delta_{i2}+\cdots +\delta'_{jn}$$
In particular, if we glue $T_i$ to $T'_j$, we need to substitute any occurrence of $\delta_i$ in $p_{T}$ and any occurrence of  $\delta'_j$ in $p_{T'}$ with $\delta_i+\delta'_j$.
\end{itemize}.
\end{prop}

\begin{remark}Some remarks about the above proposition:
\begin{itemize}
\item Compatibility under tangle addition means that we can decompose a knot as a sum of elementary tangles and recover the MAIP of the knot by adding the single pieces.
\item We suppressed the variables of the polynomial in the above for readability; clearly the ordering of the component of $T\otimes T'$ and the associated variables must be consistent on both sides of the equality $p_{T\otimes T'}=p_t+ p_{T'}$. We explicitly list these natural adjustments in the proof below. 
\item The author is a little underwhelmed by the necessity of the second condition, for it means we cannot hope to recover $p_{T\circ T'}$ from the simplified expressions for $p_T, p_{T'}$. At the same time, the $\delta_j$ are necessary to relate the weights to the intersection pairing of some special homology classes (see next section), so the author believes that the inconvenience over composition is compensated by the geometric interpretation of the weights.
\end{itemize}
\end{remark}

\begin{example}
\label{ex4}
The two tangles from Fig. \ref{example2} and \ref{example3} can be stacked on top of each other, and the result is pictured in Fig. \ref{compositionexample}. The bottom tangle $T$ has $\delta_1=-1$, $\delta_2=1$ and $\delta_3=0$, while the top tangle $T'$ has $\delta'_1=-1$, $\delta'_2=1$, and $\delta'_3=0$. 

The glueing rules from Prop. \ref{prop3} say we must set $d_3=c_1-1$, $c_3=d_3=c_1-1$, $c_2=d_1-1$ and $d_2=c_2+1=d_1$, replace $t'_3, t_3$ with $t_1$ and $t_2, t'_2$ with $t'_1$, and replace any of $\delta_2, \delta_1', \delta'_2$ with $\delta_2+\delta'_1+\delta'_2=1$ and any of $\delta_1, \delta_3, \delta'_3$ with $\delta_1+\delta_3+\delta'_3=-1$
If we take the two polynomials from Examples \ref{ex2} and \ref{ex3} and add them up we get

$$t_1^{\delta_1}(t_1^{c_1-(c_1-1)}-1)+t_1^{\delta_2}(t_1^{(c_1-1)-(c_2+1)}-1)+(t'_1)^{\delta'_3}((t'_1)^{d_1-(d_3+1)}-1)-(t'_2)^{\delta'_3}((t'_2)^{d_2-d_3}-1)$$ 
Applying the identifications prescribed by the proposition, this expression turns into
$$t_1^{-1}(t_1-1)+t_1(t_1^{(c_1-1)-d_1)}-1)+(t'_1)^{-1}((t_1')^{d_1-c_1}-1)-(t'_1)^{-1}((t_1')^{d_1-(c_1-1)}-1)$$
Simplification yields the polynomial
$$1-t_1^{-1}+t_1^{c_1-d_1}-t_1+(t'_1)^{d_1-c_1-1}-t_1'-(t_1')^{d_1-c_1}+t_1'$$

This matches the polynomial you would get if you tried computing the polynomial of the tangle resulting from the composition of $T, T'$.

\begin{figure}[!h]
\centering
\includegraphics[scale=0.1]{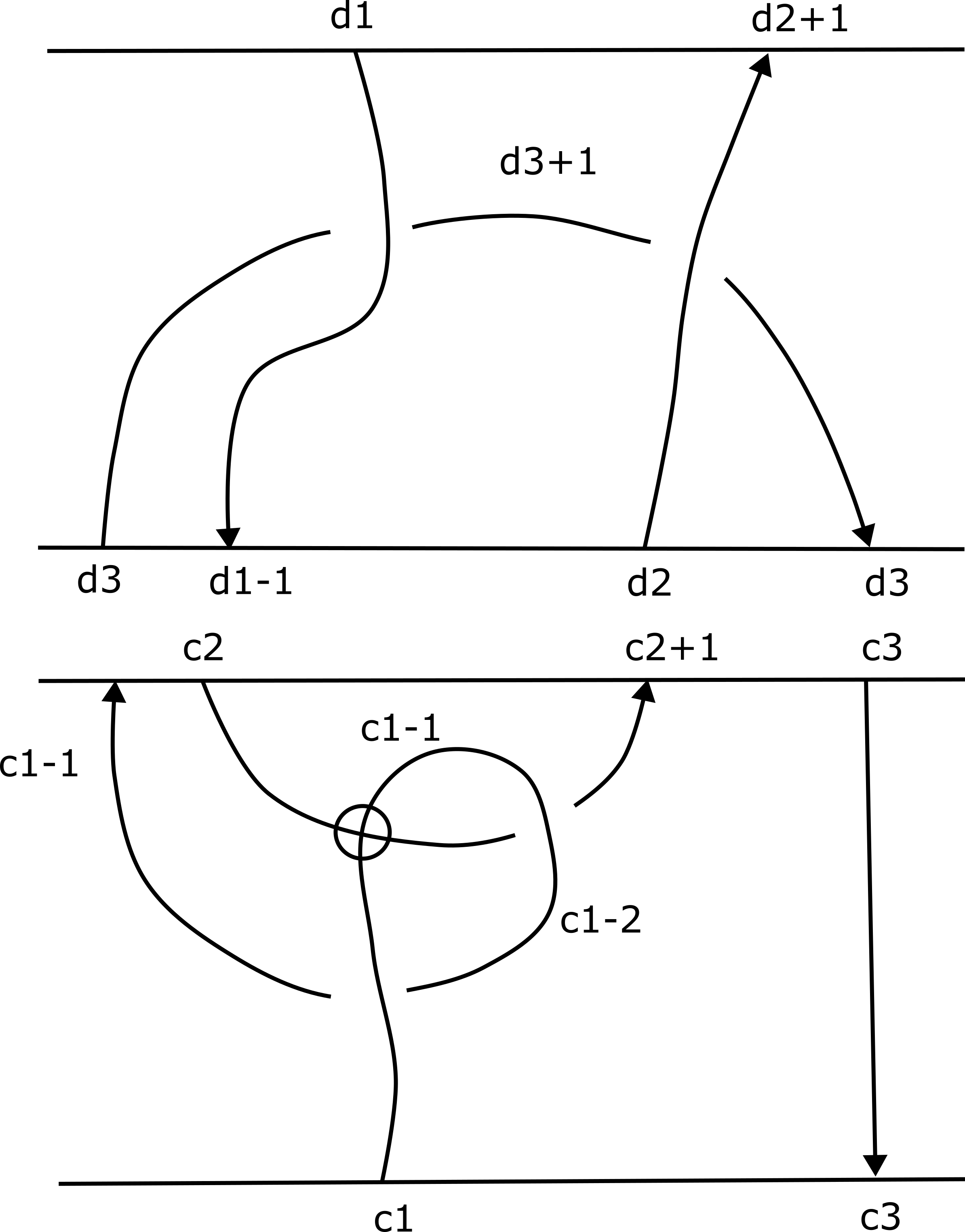}
\caption{The composition of the tangles from Fig. \ref{example2} and \ref{example3}.}
\label{compositionexample}
\end{figure}
\end{example}

\begin{remark}
Note that Example \ref{ex4} is really the same tangle as Example \ref{ex1} minus two caps/cups where the starting point of the two components lie, and the polynomials match.
\end{remark}

\begin{proof}[Proof of Prop. \ref{prop3}]
The MAIP is, by construction and by Prop. \ref{prop1}, invariant under the Turaev moves for virtual tangles. Two tangles drawn next to each other clearly don't influence each other's labeling, so if $T$ has $i$ components and $T'$ has $j$ components we get $$p_{T\otimes T'}(t_1, \ldots, t_i, t_{i+1}, \ldots, t_{i+j})=p_T(t_1,\ldots, t_i)+p_{T'}(t_{i+1}, \ldots, t_{i+j})$$

Since our tangles are oriented, when composing (i.e. stacking) tangles, the operation has to respect the orientation, and thus must glue the ending point of one component of $T/T'$ to the starting point of a component of $T'/T$. This ensures the resulting tangle is still orientable, and requiring the identification of the ending label with the starting label ensures that we get a compatible labeling on our composite tangle. 

It is possible that multiple components of $T$ and $T'$ get identified this way: $T_{i1}$ glues to $T'_{j1}$ which glues to $T_{i2}$ which glues to $T'_{j2}$ etc. The resulting component will start at the beginning of $T_{i1}$, so for consistency we should express $c'_{j1}, c_{i2}, c'_{j2}, \ldots$ all in terms of $c_{i1}$, and replace their occurrences in $p_T$ and $p_{T'}$ with the appropriate value in terms of $c_{i1}$. This process ensures that the weights of the crossings expressed in $p_T$ and $p_{T'}$ match the weights of the crossings after the composition.

Finally, let's look at the index difference before and after the composition. When two components glue together, the resulting component has an index difference equal to the sum of the index differences of the two components; replacing any occurrence of $\delta_i$ or $\delta_j'$ with $\delta_i+\delta_j'$ ensures the exponent shift caused by $t_k^{\delta_i}$ in $p_T$ or $t_k^{\delta_j'}$ in $p_{T'}$ is accounted for in $p_{T\circ T'}$.
\end{proof}

\subsection{Weights as intersection indices}
\label{homolosection}
The aim of this section is to prove that we can recover the MAIP for tangles (more specifically, the weights at each crossing) from the intersection index of some chosen homology classes related to the tangle.

\begin{defin}
Let $d$ be a self-crossing of a component $T_i$ of $T$. Define its weight as the expression $W^h(d)=(D\cdot D_s)$, where $D_s$ is the component after smoothing the crossing $c$ that contains the starting point of $T_i$, $D$ is the rest of the diagram (other components included) and $\cdot$ is the homological intersection pairing (see Remark \ref{homolopairing}).

Let $d$ be a mixed crossing between two long components $T_i, T_j$ with starting labels $c_i, c_j$ and index differences $\delta_i, \delta_j$ respectively, assuming $T_i$ goes over $T_j$. Smooth the crossing $d$, let $D_\pm$ be the left/right half of this new component, and $D'$ the rest of the diagram; define $W^h(d)=(c_i-c_j)+(D'\cdot D_\pm)$ where the choice of $\pm$ matches the sign of the crossing.

\begin{figure}[!h]
\centering
\includegraphics[scale=0.13]{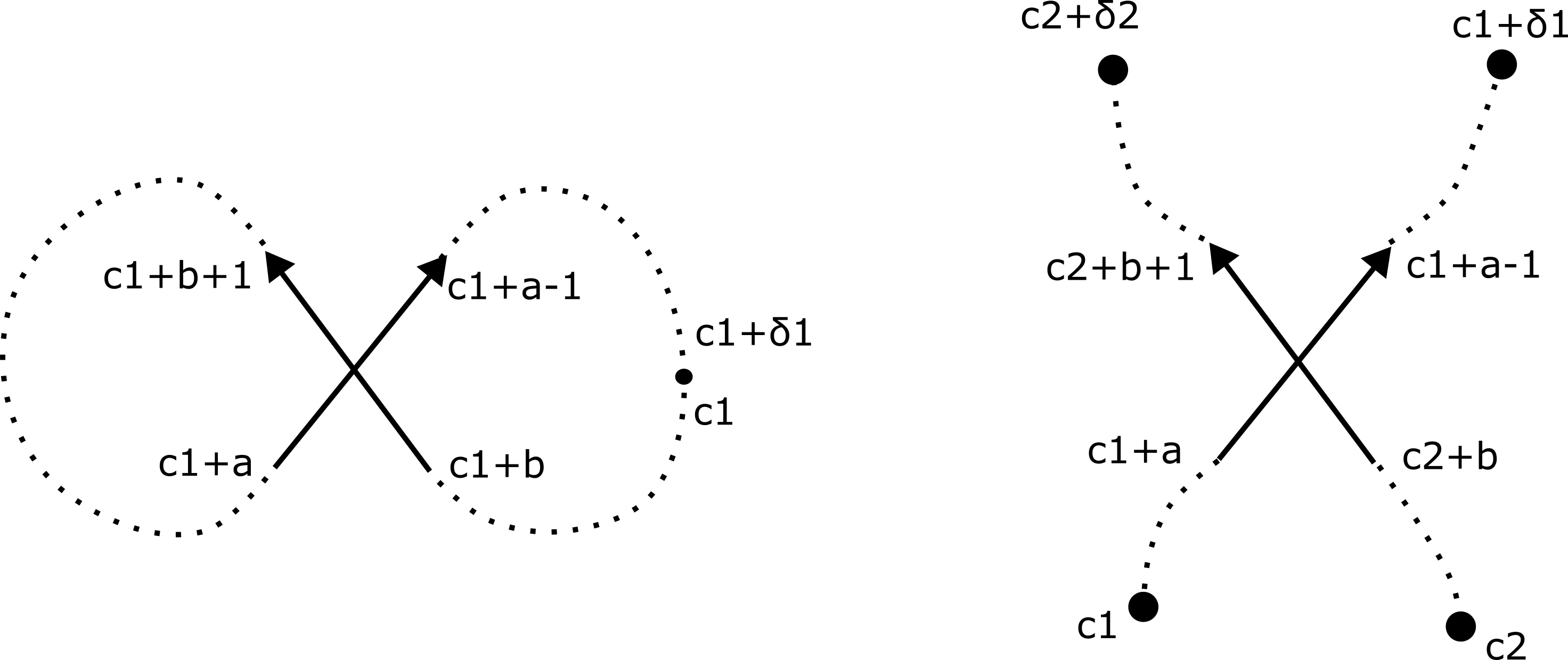}
\caption{A representation of the self-crossing and the two long components cases.}
\label{scopenhomolo}
\end{figure}

Finally, let $d$ be a mixed crossing between components $T_i, T_j$, at least one of which is closed, with starting labels $c_i, c_j$ and index differences $\delta_i, \delta_j$ respectively, with $T_i$ going over $T_j$ at the chosen crossing. Create a ``bridge'' (see Fig. \ref{closedhomolo}) connecting the strands right after the starting points of $T_i, T_j$ that preserves the orientation of both components, creating virtual crossings anytime you pass over another strand. Smooth the crossing $d$, call $D'_\pm$ be the left/right half of this new component, and $D''$ the rest of the diagram; define $W^h(d)=(c_i-c_j)+(D''\cdot D'_\pm)$ where the choice of $\pm$ matches the sign of the crossing.

\begin{figure}[!h]
\centering
\includegraphics[scale=0.13]{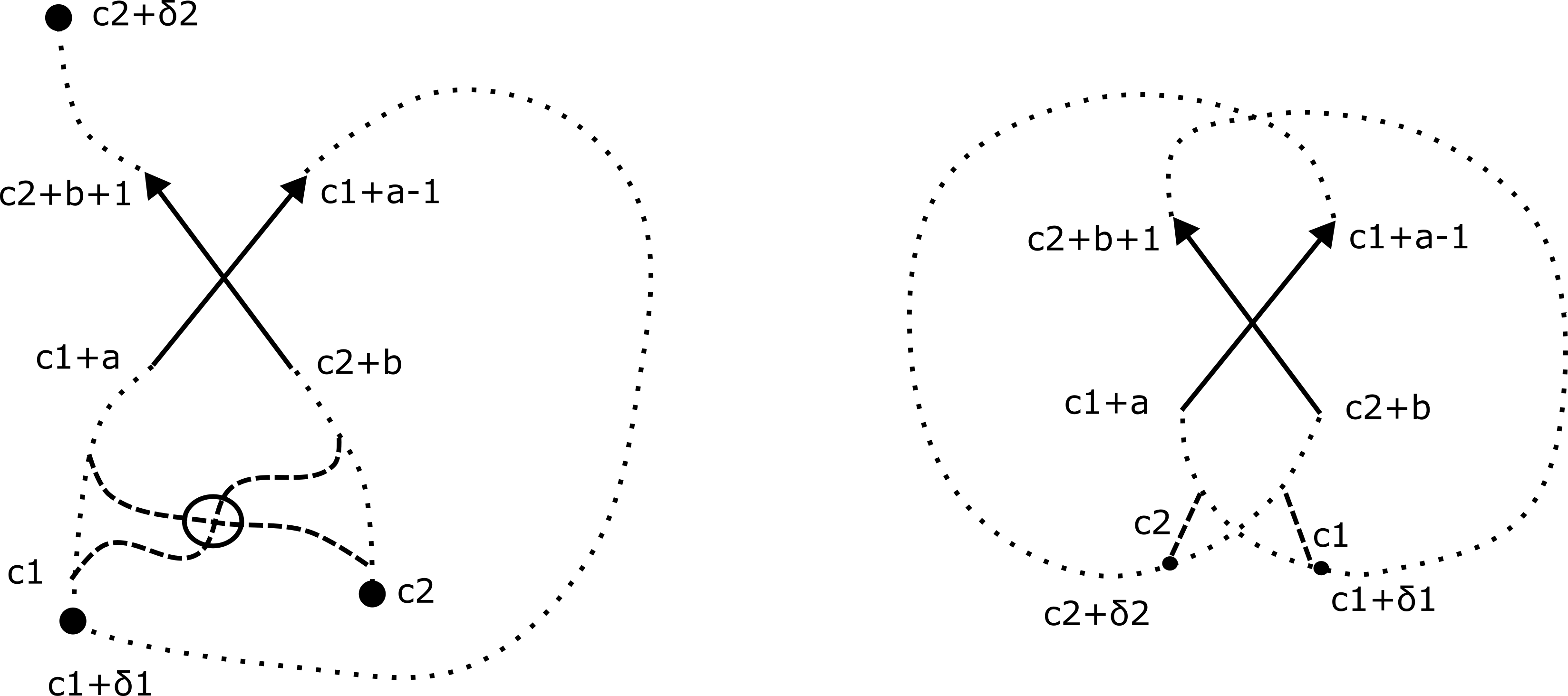}
\caption{The case when one or both of the components are closed. The dashed segments are the ``bridges'' we create..}
\label{closedhomolo}
\end{figure}
\end{defin}

\begin{prop}
\label{prop2}
For any virtual tangle and any classical crossing $d$ of said tangle, $W(d)=\pm(W^h(d)-\delta_j)$, where $W(d)$ is the MAIP weight defined in section \ref{maiptangles} and the sign is taken to be $+$ except if $d$ is a self-crossing that's an early undercrossing of the component.
\end{prop}

\begin{cor}
We can express the MAIP in terms of the homological weights as \[\begin{split}p_T(t_1,\ldots, t_n)&=\displaystyle\sum_{\substack{c\in T_i\\c\text{ early undercr.}}}sgn(c)t_i^{-W^h(d)+2\delta_i}+\sum_{\substack{c\in T_i\\ c\text{ early overcr.}}}sgn(c)t_i^{W^h(d)}+\\&+\sum_{\substack{c\in T_i\text{ over }T_j\\T_i\neq T_j}}sgn(c)t_i^{W^h(d)}-\sum_{c\in T_i\text{ over }T_j} sgn(c)t_i^{\delta_j}\end{split}\] 
\end{cor}

\begin{remark}
The ``bridges'' we introduced in the case of one or two closed components are necessary to ensure we get two homology classes after the smoothing instead of one. If $\delta_i=\delta_j=0$, these ``bridges'' can actually be inserted anywhere along the component as long as they connect two segments whose label is the same as the starting label; for convenience and consistency with the $\delta_i, \delta_j\neq 0$ case we chose to place them at the very beginning of our components.
\end{remark}

\begin{proof}[Proof of Prop. \ref{prop2}]
The key observation is that the intersection pairing at every crossing (pictured earlier in Fig. \ref{homolopairingfig}) changes in the same way that the labeling does when we take the special cycle to be the second cycle and the rest of the diagram to be the first cycle. Every time our preferred cycle goes past an oriented crossing in the bottom left to top right direction, the intersection pairing has value $-1$ and our label changes by $-1$, while if the cycle goes past a crossing in the bottom right to top left direction, both the intersection pairing and the label change have value $+1$. Hence we can simply keep track of the labeling change in a specific portion of our diagram to compute $D'\cdot D'_{\pm}$

For the case of a self-crossing, consider the sample case pictured in Fig. \ref{scopenhomolo}. Since this is a self-crossing, the component needs to close up, as denoted by the dotted lines. Since the homological pairing is the same as the labeling changes, we just need to keep track of how the label changes to compute the pairing. On the left side we must go from $c_1+b+1$ to $c_1+a$, so the net change is $c_1+a-(c_1+b+1)=a-b-1$. Similarly, along the right side the label changes by $c_1+b-(c_1+a-1)=b-a+1$. If the starting point of the component is on the left side, that net change also gains a $+\delta_1$ term, otherwise the $+\delta_1$ gets applied to the other side.
Now it's just a matter of combinatorics: $W^h(d)$ is always defined by picking the side that contains the starting point (and thus the $+\delta_1$), and it is easy to check that regardless of which half contains the starting point we have $W^h(d)=W(d)+\delta_1$ if we have an early overcrossing at $d$, and $W^h(d)=-W(d)+\delta_1$ if $d$ is an early undercrossing.

For a mixed crossing of two long components (also pictured in Fig. \ref{scopenhomolo}), knowing the starting and ending values of both components and the labels at the crossing, we can see that the homological intersection pairing yields $a+(\delta_2-b-1)$ on the left side, and $b+(\delta_1+a-1)$ on the right side. Comparing it to the expected values of $W(d)$, we see that we are only missing the contribution $\pm(c_1-c_2)$ (depending on which component goes on top), which is taken care of by the definition of $W^h(d)=(c_i-c_j)+(D'_\pm\cdot D')$.

Finally, for the case of a mixed crossing involving some closed components, the situation after the ``bridge'' was added is pictured in Fig. \ref{closedhomolo}. Since the ``bridge'' only adds virtual crossings, it does not affect the value of the homological intersection pairing, as the two strands lie on different handles in $F$ (and thus do not cross). After the adjustment and the smoothing, the homological pairing for the left cycle has value $a+(\delta_2-b-1)$ while the right cycle has value $b+(\delta_1-a+1)$. Again, we are only missing the contribution $\pm(c_i-c_j)$ to turn these into the expected value of $W(d)$.
\end{proof}

\section{Conclusion and future work}

This paper achieved the goal outlined by \cite{multivariableaip} to define an invariant compatible with tangle composition. In doing so, even more questions started coming to light. How does this invariant relate to other index polynomials, like the Henrich invariants, the Cheng-Gao polynomial, or the multi-variable Wriggle polynomial? Is this polynomial a concordance invariant? Is there a way to treat every component as if it were a long component, or any crossing in a consistent manner? Can we use similar techniques to define a multi-variable version of the Sawollek polynomial? The author hopes to answer these questions in future work.

We would also find some computational way of calculating this invariant (and the invariants it generalized). The author has an early version of Python code to compute index polynomial invariants available on his personal website, and is currently working on adding the ability to treat virtual tangles to the code. While it does not immediately relate to this paper, comments on the code are also welcome.

\bibliographystyle{amsalpha}
\bibliography{fullbibliography}

\providecommand{\bysame}{\leavevmode\hbox to3em{\hrulefill}\thinspace}
\providecommand{\MR}{\relax\ifhmode\unskip\space\fi MR }
\providecommand{\MRhref}[2]{%
  \href{http://www.ams.org/mathscinet-getitem?mr=#1}{#2}
}
\providecommand{\href}[2]{#2}
\begin{thebibliography}{Kau18}

\bibitem[JM13]{jacksonmoffatt}
D.~M. Jackson and I.~Moffatt, \emph{An introduction to quantum and vassiliev
  knot invariants}, Springer Cham, 2013.

\bibitem[Kau99]{virtualknottheory}
L.~H. Kauffman, \emph{Virtual knot theory}, European J. Comb \textbf{20}
  (1999), no.~7, 663--690.

\bibitem[Kau13]{affineindexpolynomial}
\bysame, \emph{An affine index polynomial invariant of virtual knots}, Journal
  of Knot Theory and Its Ramifications \textbf{22} (2013), no.~04, 1340007.

\bibitem[Kau18]{virtualknotcobordismaffineindex}
Louis~H. Kauffman, \emph{Virtual knot cobordism and the affine index
  polynomial}, Journal of Knot Theory and Its Ramifications \textbf{27} (2018),
  no.~11, 1843017.

\bibitem[Pet20]{multivariableaip}
Nicolas Petit, \emph{The multi-variable affine index polynomial}, Topology and
  its Applications \textbf{274} (2020), 107145.

\bibitem[Tur90]{turaevtanglemoves}
V.~Turaev, \emph{Operator invariants of tangles, and r-matrices}, Mathematics
  of the USSR-Izvestiya \textbf{32} (1990), no.~2, 411--444.

\end{thebibliography}

\end{document}